\documentclass{article}
\usepackage{amsmath,amsthm,amsfonts,amssymb}
\usepackage{graphicx}
\usepackage{xcolor}
\usepackage{authblk}

\newtheorem{theorem}{Theorem}[section]
\newtheorem{remark}[theorem]{Remark}

\newtheorem{lemma}[theorem]{Lemma}

\newtheorem{definition}[theorem]{Definition}

\newtheorem{corollary}[theorem]{Corollary}
\newcommand{\e}{\mathrm{e}}
\newcommand{\MD}{\mathrm{MD}}
\newcommand{\SMD}{\mathrm{SMD}}
\newcommand{\whp}{w.h.p.}
\newcommand{\eps}{\epsilon}
\newcommand{\symdef}{\triangle}
\newcommand{\Hmax}{H_{\mathrm{max}}}
\newcommand{\Hrow}{H_{\br}}

\newcommand{\br}{\mathbf{r}}
\newcommand{\Diam}{\mathrm{Diam}}
\newcommand{\Bin}{\textrm{Bin}}
\newcommand{\Exp}{\mathbb E}

\usepackage[colorlinks=true,allcolors=blue]{hyperref}

\title{The Metric Dimension of Sparse Random Graphs}

\author{
Josep D\'iaz \textsuperscript{1*} \quad
Harrison Hartle \textsuperscript{2$\dagger$} \quad
Cristopher Moore \textsuperscript{2$\ddagger$}
}

\begin{document}

\maketitle

\begin{abstract}
In 2013, Bollob\'as, Mitsche, and Pra\l at gave upper and lower bounds for the likely metric dimension of random Erd\H{o}s-R\'enyi graphs $G(n,p)$ for a large range of expected degrees $d=pn$. However, their results only apply when $d \ge \log^5 n$, leaving open sparser random graphs with $d < \log^5 n$. Here we provide upper and lower bounds on the likely metric dimension of $G(n,p)$ from just above the connectivity transition, i.e., where $d=pn=c \log n$ for some $c > 1$, up to $d=\log^5 n$. Our lower bound technique is based on an entropic argument which is more general than the use of Suen's inequality by Bollob\'as, Mitsche, and Pra\l at, whereas our upper bound is similar to theirs.
\end{abstract}

\footnotetext{\textsuperscript{1} Universitat Polit\`ecnica de Catalunya, Barcelona, Spain}
\footnotetext{\textsuperscript{2} Santa Fe Institute, Santa Fe, NM, USA}
\footnotetext{\textsuperscript{*} diazcort@gmail.com}
\footnotetext{\textsuperscript{$\dagger$} hhartle@santafe.edu}
\footnotetext{ \textsuperscript{$\ddagger$} moore@santafe.edu}

\section{Introduction}

Let $G=(V,E)$ be a graph with $|V|=n$ vertices. Let $d(u,v)$ denote the shortest path distance in $G$. If $u,v \in V$ are distinct and $w \in V$, we say that \emph{$w$ separates $u,v$} if $d(u,w) \ne d(v,w)$. We say that a set $W\subset V$ separates $u,v$ if some $w \in W$ does so, and we call $W$ a \emph{separator} if it separates all distinct pairs. (We will sometimes call $W$ the set of \emph{landmark} vertices.) The \emph{metric dimension} of $G$, denoted $\MD(G)$, is the size of the smallest separator $W$.

This problem was defined independently by Slater~\cite{Slater75} and Harary-Melter~\cite{HararyM76} more than 50 years ago. From the complexity point of view, determining $\MD(G)$ is NP-complete for general graphs~\cite{GareyJohnson,Khuller96}, planar graphs (even with bounded degree)~\cite{diaz17}, split and bipartite graphs~\cite{Epstein12}, and interval and permutation graphs (even with diameter 2)~\cite{Foucaud17}. For general graphs, $\MD(G)$ can be approximated in polynomial time within a factor of $O(\log n)$~\cite{Khuller96,Hauptmann12}, but not within an $o(\log n)$ factor unless $\mathrm{P}=\mathrm{NP}$~\cite{Beerliova}. Even an $(1-\eps) \log n$ approximation in polynomial time would imply that $\mathrm{NP} \subseteq \mathrm{DTIME}(n^{\log \log n})$~\cite{Hauptmann12}. 

On the positive side, there are polynomial time algorithms for $\MD(G)$ for trees~\cite{Khuller96}, chain graphs~\cite{Fernau15}, 
co-graphs~\cite{Epstein12}, and outer-planar graphs~\cite{diaz17}. There also have been a number of works dealing with the parameterized complexity of $\MD(G)$, 
producing some FPT algorithms under certain parametrizations; see~\cite{Belmonte15} and the bibliography there. 
In the already cited~\cite{Khuller96} the authors characterized graphs with metric dimensions $1$ and $2$, and provided a general bound that appears as Equation~\eqref{eq:khuller} below. For further results on the combinatorial nature of the problem, see~\cite{Chartrand00}.

Given the difficult nature of the metric dimension problem on many types of graphs, a natural  question is to find the typical value of $\MD(G)$ for various random graphs.For the Erd\H{o}s-Rényi-Gilbert random graph model $G(n,p)$,
Bollob\'as, Mitsche and Pra\l at~\cite{BollMitPra} largely determined the likely value of $\MD(G)$ quite precisely for a wide range of $p$; we review their results as Theorem~\ref{thm:bollmitpra} in the next Section. However, their results do not apply unless the average degree $d=np$ is at least $\log^5 n$, leaving sparser random graphs as an open problem.

In this paper we extend the results of~\cite{BollMitPra} to sparser random graphs, and almost all the way down to the connectivity transition at $d=\log n$. Specifically, we give upper and lower bounds on $\MD(G)$ whenever $d \ge c \log n$ for any $c = c(n) > 1$. In exchange for this wider range of $d$, our lower bound is somewhat looser than those of~\cite{BollMitPra}.

We obtain our results with two techniques. Our lower bound is information-theoretic, using the entropy of the matrix of pairwise distances to show that a certain number of landmarks are needed to distinguish all pairs of vertices. This is significantly simpler than the lower bound method in~\cite{BollMitPra} which uses Suen's inequality.

We note that \'Odor and Thiran~\cite{odor-thiran,Odor22}  recently considered an adaptive version of the metric dimension, where there is a hidden vertex $v^\star$ and a player can ask a series of adaptive queries $v_1,\ldots,v_\ell$, each of which yields the distance $d(v_i,v^\star)$. They define the \emph{sequential metric dimension} $\SMD(G)$ of a graph $G$ as the smallest $\ell$ such that the player can always succeed in identifying $v^\star$. Clearly $\SMD(G) \le \MD(G)$, since the usual metric dimension yields a non-adaptive strategy. They prove lower bounds on $\SMD(G(n,d/n))$  and find that for sufficiently large $d$ the ratio between the $\MD$ and $\SMD$ is at most a constant. However, like~\cite{BollMitPra}, their results only apply when $d \ge \log^5 n$ due to their use of Suen's inequality rather than entropy. In fact our entropic lower bound on $\MD(G(n,p))$ is also a lower bound on $\SMD(G(n,p))$.

Our upper bound technique is identical to that of~\cite{BollMitPra}. Given $u$ and $v$, consider the ``shells'' $V_t(u)$ and $V_t(v)$ of vertices whose shortest path distance to $u$ or $v$ respectively is exactly $t$. If the symmetric difference $V_t(u) \symdef V_t(v)$ is large for some $t$, then there are many vertices $w$ that separate $u$ and $v$. If this is true for all pairs $u,v$, we can construct a separator $W$ simply by choosing landmarks uniformly at random.

Thus both our upper and lower bound work by inductively bounding the size of the shells $V_t$ for $t=1, 2, 3, \ldots$  It turns out that almost all of the fluctuations in $|V_t|$ occur at the first step of this process, namely in the degree $|V_1|$ of $u$ or $v$. Proving our bounds for $d=c \log n$ for all constant $c > 1$, as opposed to some larger constant, requires a concentration inequality for the binomial distribution which is sharper than the standard Chernoff bound. Moreover, in the regime $d=O(\log n)$, the vertex degrees really do fluctuate multiplicatively, between $\alpha d$ and $\beta d$ for some constants $0 < \alpha < 1 < \beta$. This introduces an additional multiplicative gap in our results. This may be unavoidable, since some pairs of vertices really do have fewer separators than other pairs.

Throughout the paper, we assume that the reader is familiar with the basic properties of $G(n,p)$ (see for example~\cite{Jansonbook,Moorebook}). In particular, since we have $d > \log n$, we condition throughout on the event that $G(n,p)$ is connected. We also assume basic knowledge of entropy manipulation (see for example~\cite{Shannon,Radha01,cover1999elements}).

We say a series of events $E=(E_n)$ holds with high probability (\whp) if $\Pr[E_n] \to 1$ as $n \to \infty$. We use asymptotic notation in these statements as follows: ``\whp\ $f(n) = O(g(n))$'' means that there is a constant $C$ such that \whp\ $f(n) \le C g(n)$, and ``\whp\ $f(n)=o(g(n))$'' means that \whp\ $f(n) \le C g(n)$ for any constant $C > 0$. We use $\Omega$ and $\omega$ analogously, if $f(n) \ge C g(n)$ \whp\ for some, or all, constants $C > 0$.

\section{Two cases}

The behavior of the metric dimension of $G(n,d/n)$ turns out to be somewhat strange. Imagine a series of shells $V_t$ surrounding a vertex, consisting of the vertices at distance $t = 1, 2, 3 \ldots$.  Roughly speaking, we expect the number of vertices in $V_t$ to grow as $d^t$ where $d=np$ is the average degree. The question is what happens when $t$ gets close to the diameter of $G$, so that the $t$th shell $V_t$ contains a significant fraction of all $n$ vertices. This depends on how close $n$ is to an integer power of $d$. 

Following~\cite{BollMitPra} with small changes, we adopt the following notation, which we will use throughout.

\begin{definition}
\label{def:cases}
Given $d$, $n$, and $p=d/n$, let $t^*$ be the largest integer $t$ such that $d^t = o(n)$. Let 
\begin{equation}
\label{eq:gamma}
\gamma = \gamma(n) = p d^{t^*} = d^{t^*+1}/n \, .
\end{equation}
We then distinguish two cases:
\begin{itemize}
\item Case 1: $\gamma = \Theta(1)$. 
\item Case 2: $\gamma = \omega(1)$.
\end{itemize}
\end{definition}
\noindent 
Note that this definition implies that $\gamma = o(d)$.

As we will see, the key difference between Case 1 and Case 2 is whether the $(t^*+1)$st and $(t^*+2)$nd shells each contain $\Theta(n)$ vertices, with all others containing $o(n)$, or whether the $(t^*+1)$st shell contains all but $o(n)$ vertices. This in turn  determines whether the entropy of the distribution of distances from a given vertex is $\Theta(1)$ or $o(1)$, and whether a typical pair of vertices has $\Theta(n)$ or $o(n)$ separators. As a result, the metric dimension will be $\Theta(\log n)$ in Case 1 and $\omega(\log n)$ in Case 2.

As pointed out in~\cite{BollMitPra}, this causes $\MD(G)$ to undergo a curious zig-zag behavior. For a given function $d=d(n)$, as $n$ increases we briefly visit Case 1 whenever $n$ is close to an integer power of $d$. In between these integer powers, we have Case 2. Their main result, slightly rewritten to focus on the sparse case $d=o(n)$, is the following:

\begin{theorem}[\cite{BollMitPra}] 
\label{thm:bollmitpra}
Let $G=G(n,p)$ where $p=d/n$ and $d = d(n) > \log^5 n$ and $d=o(n)$. Let $t^*$ and $\gamma$ be as in Definition~\ref{def:cases}. 

In Case 1 where $\gamma=\Theta(1)$ then \whp
\[
\MD(G) = (1 \pm o(1)) \,\frac{2 \log n}{-\log q} = \Theta (\log n) \, ,
\]
where 
\[ 
q = (\e^{-\gamma})^2 + (1-\e^{-\gamma})^2 \, .
\]

In Case 2 where $\gamma=\omega(1)$,
\whp 
\[
\MD(G) 
= \Theta\!\left( \frac{\log n}{{\gamma/d} + \e^{-\gamma}} \right)
= \omega(\log n) \, .
\]
\end{theorem}

\section{Our results}

We now state our bounds. They are somewhat verbose, especially when $d=O(\log n)$. Nevertheless, they determine the likely value of $\MD(G)$ to within a constant or slowly-growing functions of $n$.

\begin{theorem}
\label{thm:main-combined}
Let $G=G(n,p)$ where $p=d/n$ and $c \log n \le d \le \log^5 n$ for some $c > 1$. Let $t^*$ and $\gamma$ be as in Definition~\ref{def:cases}. Then there are constants $0 < \alpha < 1 < \beta$ depending on $c$ such that the following holds. Let 
\[ H(y) = -y \log y -(1-y) \log (1-y) \]
be the binary entropy function and let 
\[ q(\alpha,\beta) = 
1 - 2
\,(1-\e^{-\alpha\gamma}) 
\,\e^{-\beta\gamma} \, . 
\]

In Case 1 where $\gamma=\Theta(1)$, \whp 
\begin{equation}
\label{eq:summary-case1-olog}
 (1-o(1)) \,\frac{\log n}{\max_{x \in [\alpha,\beta]} H(\e^{-x\gamma})} \le \MD(G) \le (1+o(1)) \,\frac{2 \log n}{-\log q(\alpha,\beta)}
 \, .
\end{equation}
Moreover, if $c=\omega(1)$ so that $d=\omega(\log n)$, then 
\begin{equation}
\label{eq:summary-case1-wlog}
 (1-o(1)) \,\frac{\log n}{H(\e^{-\gamma})} \le \MD(G) \le (1+o(1)) \,\frac{2 \log n}{-\log q} 
 \, ,
\end{equation}
where $q=q(1,1)=(\e^{-\gamma})^2 + (1-\e^{-\gamma})^2$ as in Theorem~\ref{thm:bollmitpra}.

In Case 2 where $\gamma=\omega(1)$, 
\whp 
\begin{equation}
\label{eq:summary-case2-olog}
 (1-o(1)) \,\frac{\log n}{-(\beta \gamma/d) \log (\beta \gamma/d) 
+ \gamma \e^{-\alpha \gamma}} 
\le \MD(D) \le 
(1+o(1)) \,\frac{\log n}{\alpha\gamma/d + \e^{-\beta\gamma}} \, ,
\end{equation}
and if $d=\omega(\log n)$ we have
\begin{equation}
\label{eq:summary-case2-wlog}
(1-o(1)) \,\frac{\log n}
{-(\gamma/d) \log (\gamma/d) 
+ \gamma \e^{-\gamma}} 
\le \MD(G) \le 
(1+o(1)) \,\frac{\log n}
{\gamma/d + \e^{-\gamma}} \, .
\end{equation}
\end{theorem}

We prove Theorem~\ref{thm:main-combined} below, breaking it into Theorem~\ref{thm:main-lower} and Theorem~\ref{thm:main-upper} for the lower and upper bounds respectively. 

We can summarize these bounds as follows. First, our upper bounds on $\MD(G)$ match those in~\cite{BollMitPra}, with the complication that when $d=O(\log n)$ we have to introduce the constants $\alpha$ and $\beta$ to handle fluctuations in vertex degrees. 
Our lower bounds are somewhat looser than those of~\cite{BollMitPra}, but have the advantage that they apply when $d \ge c \log n$ as opposed to $d \ge \log^5 n$. 

\smallskip \textbf{Case 1.} When $\gamma=\Theta(1)$, our upper and lower bounds differ by a multiplicative constant depending on $\gamma$, and \whp\ 
\[
\MD(G) = \Theta(\log n) \, .
\]
When $d=c \log n$ with $c=O(1)$ this gap also depends on $c$, since it depends on $\alpha$ and $\beta$. But since $H(y) \le \log 2$ for all $y$, we always have \whp
\[
\MD(G) \ge (1-o(1)) \log_2 n 
\, .
\]

\smallskip \textbf{Case 2.} Since $\gamma=o(d)$ and $\gamma=\omega(1)$, the denominators in~\eqref{eq:summary-case2-olog} and~\eqref{eq:summary-case2-wlog} are all $o(1)$, so \whp\ $\MD(G) = \omega(\log n)$. However, in this case our upper and lower bounds are more than a constant apart. Moreover, the scaling of $\MD(G)$ with $n$ depends on which term in the denominators of our bounds is larger, and this in turn depends on how $\gamma$ grows with $n$. The crossover occurs roughly when $\gamma=\log d$.

First suppose $d=\omega(\log n)$, and recall that $d \le \log^5 n$ so $\log d = O(\log \log n)$. Then~\eqref{eq:summary-case2-wlog} gives the following.
\begin{itemize}
\item If $d=\omega(\log n)$ and $\gamma \ge \log d$, then \whp
\[
(1-o(1)) \frac{\log n}
{-(\gamma/d) \log (\gamma/d)}
\le \MD(G) \le
(1+o(1)) \frac{\log n}
{\gamma/d} \, ,
\]
and so
\[
\Omega\!\left(
\frac{d \log n}{\gamma \log \log n} 
\right)
\le \MD(G) \le
O\!\left( \frac{d \log n}{\gamma}
\right) \, .
\]
\item If 
$d=\omega(\log n)$ and $\gamma \le \delta \log d$ for some $\delta < 1$, then \whp
\[
(1-o(1)) \frac{\log n}
{\gamma \e^{-\gamma}}
\le \MD(G) \le
(1+o(1)) \frac{\log n}
{\e^{-\gamma}} \, .
\]
In particular, if $\gamma= \delta \log d$ then
\[
\Omega\!\left(
\frac{d^\delta \log n}{\log \log n} 
\right)
\le \MD(G) \le
O(d^\delta \log n) \, .
\]
\item Combining these, if $d=\omega(\log n)$ the multiplicative gap between our upper and lower bounds is $O(\min(\log \log n,\gamma))$. 
\end{itemize}

Finally, consider Case 2 when $d=c \log n$ for some constant $c > 1$. Here the situation is complicated by the fact that $\alpha < 1 < \beta$, and that these appear in the exponents $\e^{-\alpha \gamma}$ and $\e^{-\beta\gamma}$. This creates a pair of crossovers with a range of $\gamma$ in between. From~\eqref{eq:summary-case2-olog} we have the following.

\begin{itemize}
\item 
If $d=\Theta(\log n)$ and $\gamma \ge (1/\alpha) \log d$, then \whp\ 
\[
(1-o(1)) 
\frac{\log n}
{-(\beta \gamma/d) \log(\beta \gamma/d)} 
\le \MD(G) \le
(1+o(1)) 
\frac{\log n}
{\alpha \gamma/d} 
\]
and so
\[
\Omega\!\left(
\frac{\log^2 n}{\gamma \log \log n} 
\right) 
\le \MD(G) \le
O\!\left( 
\frac{\log^2 n}{\gamma} \right) \, . 
\]
\item If $d=\Theta(\log n)$ and $\gamma = \delta \log d$ for $1/\beta \le \delta < 1/\alpha$, then \whp\ 
\[
 (1-o(1)) \,\frac{\log n}{\gamma \e^{-\alpha \gamma}} 
\le \MD(D) \le 
(1+o(1)) \,\frac{\log n}{\alpha\gamma/d} \, ,
\]
and so
\begin{equation}
\label{eq:widest}
\Omega\!\left(
\frac{\log^{1+\alpha\delta} n}{\log \log n} 
\right) 
\le \MD(G) \le
O\!\left( 
\frac{\log^2 n}{\log \log n} \right) \, . 
\end{equation}
\item If $d=\Theta(\log n)$ and $\gamma = \delta \log d$ for $\delta < 1/\beta$, then \whp\ 
\[
 (1-o(1)) \,\frac{\log n}{\gamma \e^{-\alpha \gamma}} 
\le \MD(D) \le 
(1+o(1)) \,\frac{\log n}{\e^{-\beta\gamma}} \, ,
\]
and so
\begin{equation}
\label{eq:widest2}
\Omega\!\left(
\frac{\log^{1+\alpha\delta} n}{\log \log n} 
\right) 
\le \MD(G) \le
O\!\left( 
\frac{\log^{1+\beta\delta} n}{\log \log n} \right) \, . 
\end{equation}
\end{itemize}

The multiplicative gap between our bounds is most severe in~\eqref{eq:widest} and~\eqref{eq:widest2}, where we remain uncertain about the power of $\log n$. The source of this gap is that if $d=c \log n$, then there typically exist vertices with degrees ranging from $\alpha d$ to $\beta d$ for some $\alpha < 1 < \beta$ depending on $c$. As a result, different pairs of vertices really do have different numbers of separators that distinguish them, and different vertices have different distributions of distances to the rest of the graph. This doesn't necessarily mean that $\MD(G)$ fluctuates from one realization of $G(n,d/n)$ to another, but it does make it harder to prove that it is concentrated around some typical value.

\section{An entropic lower bound on metric dimension}

We begin by proving a general lower bound on the width of a matrix whose rows are distinct and whose columns have bounded entropy. Note that this lemma applies to arbitrary deterministic matrices, although we will apply it below to a random matrix.

When the base matters, our logarithms are natural.

\begin{lemma}
\label{lem:entropy}
Let $A$ be a matrix with $n$ rows and $m$ columns with entries in a finite set $S$. For each $1 \le j \le m$ and each 
$s \in S$, let $p_j(s) = |\{i: A_{ij} = s\}| / n$ denote the frequency of $s$ in the $j$th column, and define the entropy
\[
H_j = - \sum_{s \in S} p_j(s) \log p_j(s) \, ,
\]
where $0 \log 0 := 0$. 
Suppose that $H_j \le \Hmax$ for all $1 \le j \le m$ for some $\Hmax$.  
Then if the $n$ rows of $A$ are distinct, we have
\[
m \ge \frac{\log n}{\Hmax} \, .
\]
\end{lemma}

\begin{remark}
Even though we will consider the case where each $A_{ij}$ is a random variable, namely the shortest path distance in a random graph, the entropy $H_j$ is not the entropy of that random process. Rather, $H_j$ is the entropy of the frequency distribution of the $j$th column of a fixed matrix. In statistical language, it is the entropy of the empirical distribution of that column. 

This entropic argument has the advantage that it does not require any independence in the entries of $A$. A stronger lower bound on the number of columns is proved in~\cite{odor-thiran} in the case where the entries of $A$ are independent Bernoulli variables.
\end{remark}

\begin{proof}
This is a classic information-theoretic proof. While we will phrase it in terms of random variables, the entries $A_{ij}$ of the matrix $A$ are fixed; the only randomness will come from picking a random row $i$.

For each fixed column $j \in \{1,\ldots,m\}$, consider the random variable $A_{ij}$ 
where $i$ is uniformly random in $\{1,\ldots,n\}$. Then $p_j(s)$ is the probability distribution of this variable, i.e., the empirical distribution of $s$ in column $j$, and $H_j$ is its entropy.

Now consider the random variable $\br$ consisting of a randomly chosen row of $A$, namely $\br=(A_{i,1},A_{i,2},\ldots,A_{i,m})$ where $i$ is uniformly random in $\{1,\ldots,n\}$. Let $\Hrow$ be the entropy of the distribution of $\br$. For each $1 \le j \le m$ the marginal distribution of the $j$th component $\br_j=A_{ij}$ is distributed according to $p_j$. Thus by the sub-additivity property of entropy we have
\begin{equation}
\label{eq:hrow1}
\Hrow \le \sum_{j=1}^m H_j \le m \Hmax \, .
\end{equation}
since $\max_j H_j \le \Hmax$ by assumption. If the $n$ rows are distinct, then each row appears with probability $1/n$ and $\Hrow = \log n$. 
Combining this with~\eqref{eq:hrow1} completes the proof.
\end{proof}

We use this entropic argument to metric dimension as follows. Consider a graph $G=(V,E)$, and let $W \subseteq V$ be a set of vertices. Let $n=|V|$ and $m=|W|$. Then let $A$ be the matrix of shortest path distances $A_{v,w} = d(v,w)$, where each row corresponds to some $v \in V$ and each column corresponds to some $w \in W$. If $W$ is a separator, then the $n$ rows of $A$ are distinct. Applying Lemma~\ref{lem:entropy} then yields a lower bound on $m$, and thus on $\MD(G)$.

We use the following notation for the ``shells'' of vertices at a given distance from $w$:

\begin{definition}
\label{definition:shells}
Let $G=(V,E)$ with $|V|=n$ and diameter $\Diam$. For each $w \in V$ and each $0 \le t \le \Diam$, let
\[
V_t(w) = \{ v \in V : d(v,w) = t \} 
\]
be the set of vertices at distance exactly $t$ from $w$. We will also write
\[
V_{\le t}
= \bigcup_{t'=0}^{t} V_{t'} 
\quad \text{and} \quad
V_{>t} 
= V \setminus V_{\le t}
= \bigcup_{t'=t+1}^{\Diam} V_{t'} 
\, . 
\]
\end{definition}

Now we state an immediate corollary of Lemma~\ref{lem:entropy}, which gives a lower bound on $\MD(G)$ for all graphs:

\begin{corollary}
\label{cor:main}
For any graph $G=(V,E)$ with $|V|=n$ and diameter $\Diam$, and for a given $w \in V$, let
\begin{equation}
\label{eq:Hw}
H_w = - \sum_{t=0}^\Diam \frac{|V_t(w)|}{n} \log \frac{|V_t(w)|}{n} \, . 
\end{equation}
Then if $H_w \le \Hmax$ for all $w \in V$, we have
\[
\MD(G) \ge \frac{\log n}{\Hmax} \, .
\]
\end{corollary}

\begin{remark}
This is also a lower bound on the sequential metric dimension $\SMD(G)$ considered in~\cite{odor-thiran}, since the player can gain at most $\Hmax$ bits of information with each query, and they need $\log n$ bits to identify the hidden vertex.
\end{remark}

As a simple application, for any graph $G$ with diameter $\Diam$ we have $\Hmax \le \log |S| = \log (\Diam+1)$ since $S=\{0,1,\ldots,\Diam\}$. Thus, if $G$ has $n$ vertices we have
\begin{equation}
\label{eq:bound-simple}
\MD(G) \ge \frac{\log n}{\log (\Diam+1)} \, .
\end{equation}
We can improve this slightly by focusing on the $(n-m) \times m$ matrix of distances between $V \setminus W$ and $W$. Then $S=\{1,\ldots,\Diam\}$ and 
\begin{equation}
\label{eq:bound-pessimist}
\MD(G) \ge \frac{\log (n-\MD(G))}{\log \Diam} \, .
\end{equation}
This is equivalent to a bound from~\cite{Khuller96},
\begin{equation}
\label{eq:khuller}
n \le \Diam^{\MD(G)} + \MD(G) \, .
\end{equation}

However, the lower bounds~\eqref{eq:bound-simple} and~\eqref{eq:bound-pessimist} are quite pessimistic. We can only have $\Hmax = \log (\Diam+1)$ if, for some $w$, the distribution of distances $d(v,w)$ is uniform between $0$ and $\Diam$. This holds for a path graph of $n$ vertices, where $\Diam = n-1$ and these bounds give $\MD(G) \ge 1$. Indeed, for path graphs we have $\MD(G)=1$ since a single vertex at either end is a separator.

But in an expander, in particular \whp\ in $G(n,p)$, this is far from true: only a small fraction of vertices have distances much less than the diameter. In fact \whp\ for any $w$ the distribution of distances $d(v,w)$ is concentrated on one or two values. (This was already known for $d=\omega(\log n)$; we give a proof here that includes the case $d=O(\log n)$.) In that case we have $\Hmax=O(1)$, improving the lower bound from $\MD(G) \ge (\log n)/(\log \Diam) = \Omega(\log n / \log \log n)$ to $\MD(G) = \Omega(\log n)$.

The following two lemmas overlap with classic results about expansion in $G(n,p)$ (e.g.~\cite{bollobas1981diameter,chung2001diameter}) but we tune them to our needs and give simple self-contained proofs. We use $\deg w$ to denote the degree of a vertex $w$.

\begin{lemma}
\label{lem:expand}
Let $G = G(n,p)$ where $p=d/n$ and $d \ge c \log n$ for some $c > 1$. Define $V_t(w)$ as in Corollary~\ref{cor:main}. Then there are constants $\alpha, \beta > 0$ depending only on $c$ such that the following holds \whp: for all $w \in V$,
\begin{equation}
\label{eq:degree-bound}
(1-o(1)) \,\alpha d \le \deg w \le (1+o(1)) \,\beta d \, ,
\end{equation}
and, for all $1 \le t \le t^*$ with $t^*$ as in Definition~\ref{def:cases}, 
\begin{equation}
\label{eq:Nupper-lower-expand0}
|V_t(w)| = (1 \pm o(1)) \,d^{t-1} \deg w \, .
\end{equation}
Thus \whp\ for all $w \in V$ and all $t \le t^*$, 
\begin{equation}
\label{eq:Nupper-lower-expand}
(1-o(1)) \,\alpha d^t 
\le |V_t(w)| 
\le (1+o(1)) \,\beta d^t \, .
\end{equation}
Moreover, if $c = \omega(1)$, i.e., if $d=\omega(\log n)$, then $\alpha=\beta=1$. 
\end{lemma}

\noindent We prove this lemma in Section~\ref{sec:proof-expand}.

Now that we know how the shells $V_t$ grow exponentially with $t$, the question is how this process ends in the last few shells. First we use a standard argument to show that the diameter of $G$ is concentrated on a few values.

\begin{lemma}
\label{lem:diameter}
Let $G = G(n,p)$ where $p=d/n$ and $d \ge c \log n$ for some $c > 1$. Let $t^*$ and $\gamma$ be defined as in Definition~\ref{def:cases}. Then \whp\
\begin{align*}
\Diam(G) &\le t^* + 3 \quad
\text{in Case 1 where $\gamma=O(1)$} \\
\Diam(G) &\le t^* + 2 \quad
\text{in Case 2 where $\gamma=\omega(1)$.} 
\end{align*}
\end{lemma}

\begin{remark}
Since $\Diam(G) \ge t^*$, Lemma~\ref{lem:diameter} shows that the diameter is concentrated on at most four values. Sharper bounds are available for larger $d$~\cite{bollobas1981diameter,chung2001diameter}, and moreover the same proof shows that $\Diam(G) \le t^*+2$ in Case 1 if $\alpha^2 \gamma > 2$, but these bounds are enough for our purposes.
\end{remark}

\begin{proof}
Focusing on Case 2 first, we have $d^{t^*+1} = \gamma n = \omega(n)$. Lemma~\ref{lem:expand} then implies that \whp\ for any pair of vertices $u, v$ and any $1 \le s,t \le t^*$ with $s+t=t^*+1$, 
\[
|V_s(u)| |V_t(v)| 
\ge (1-o(1)) \,\alpha^2 d^{s+t}
= \Theta(d^{t^*+1}) 
= \omega(n) 
\]
If $V_s(u)$ and $V_t(v)$ intersect, then $d(u,v) \le s+t = t^*+1$. If they are disjoint but there is an edge between them, then $d(u,v) \le s+t+1 = t^*+2$. The probability there is no such edge is
\[
(1-p)^{|V_s(u)| |V_t(v)|}
= (1-d/n)^{\omega(n)} 
\le \e^{-\omega(d)} \, .
\]
Since $d > \log n$, this is smaller than $n^{-C}$ for any constant $C$. Taking the union bound over all $n^2$ pairs $u,v$ thus shows that \whp\ $\Diam(G) \le s+t+1 \le t^*+2$, completing the proof for Case 2. 

For Case 1, we use the same argument but with $s+t=t^*+2$, in which case from $d^{s+t} = d^{t^*+2} = \omega(n)$.
\end{proof}

Next, the following lemma shows that in Case 1, two shells occupy almost the entire graph, namely $V_{t^*+1}$ and $V_{t^*+2}$. That is, for all $w$, the distance $d(w,v)$ takes just two different values for almost all $v$. In Case 2, just one shell $V_{t^*+1}$ dominates the graph, so $d(w,v)$ takes just one value for almost all $v$. As a result, the entropy of distances is $O(1)$ and $o(1)$ in Case 1 and Case 2 respectively.

\begin{lemma}
\label{lem:last-few}
Let $G = G(n,p)$ where $p=d/n$ and let $c \log n \le d \le \log^5 n$ with $c > 1$. Let  $V_t(w)$, $t^*$, $\alpha$, $\beta$, and $\gamma$ be as above. 
In Case 1 where $\gamma = \Theta(1)$, \whp\ the following holds for all $w \in V$:
\begin{gather}
1-\e^{-\alpha \gamma} - o(1)
\le \frac{|V_{t^*+1}(w)|}{n}
\le 1-\e^{-\beta \gamma} + o(1) \, , 
\label{eq:case1-tplus1} \\
\e^{-\beta \gamma} - o(1) 
\le \frac{|V_{t^*+2}(w)|}{n} 
\le \e^{-\alpha \gamma} + o(1) \, , 
\label{eq:case1-tplus2} \\
\frac{|V_{t^*+1}(w)|}{n} + \frac{|V_{t^*+2}(w)|}{n} = 1-o(1) \, . 
\label{eq:case1-rest}
\end{gather}
In Case 2 where $\gamma = \omega(1)$, \whp\ for all $w \in V$:
\begin{gather}
\frac{|V_{t^*+1}(w)|}{n}
\ge 1 - (1+o(1))
\left( \frac{\beta\gamma}{d} 
+ \e^{-\alpha \gamma} \right) 
\, .
\label{eq:case2-tplus1}
\end{gather}
\end{lemma}

\begin{remark}
We restrict our statement to $d \le \log^5 n$ for convenience and because when $d \ge \log^5 n$ the results of~\cite{BollMitPra} apply. We claim that in fact our proof holds for any $d = O(n/\log n)$.
\end{remark}

\begin{proof}
We write $V_t$ for $V_t(w)$ for simplicity. 
Recall that $V_{\le t^*}$ is the union of all the shells whose sizes are upper and lower bounded by Lemma~\ref{lem:expand}. Since $|V_t|$ grows by a factor of $\Theta(d)$ at each step, $|V_{\le t^*}|$ is a geometric sum dominated by its largest term $|V_{t^*}|$. Thus
\begin{equation}
\label{eq:geometric}
|V_{\le t^*}|
= \sum_{t=0}^{t^*} |V_t|
= (1+O(1/d)) \,|V_{t^*}|
= O(d^{t^*}\!/n)
= o(1)
\, .
\end{equation}

Now let $n'$ be the number of vertices outside these shells,
\begin{equation}
\label{eq:n-prime}
n' 
= |V_{>t^*}|
= n - |V_{\le t^*}|
= (1-o(1)) n
\, . 
\end{equation}
Conditioned on $n'$ and $|V_{t^*}|$, we have 
\begin{equation}
\label{eq:tplus1}
|V_{t^*+1}| \sim \Bin(n', p')  
\end{equation}
where $p'$ is the probability a given vertex in $V_{>t^*}$ is a neighbor of at least one vertex in $V_{t^*}$. That is,
\begin{equation}
\label{eq:p-prime}
p' = 1-(1-p)^{|V_{t^*}|} 
= 1 - \e^{-p |V_{t^*}|} + O(p^2 |V_{t^*}|) \, ,
\end{equation}
where we used the Taylor series. We can rewrite~\eqref{eq:Nupper-lower-expand} for $t=t^*$ as
\begin{equation}
\label{eq:rewrite}
(1-o(1)) \,\alpha \gamma 
\le p |V_{t^*}| 
\le (1+o(1)) \,\beta \gamma \, .
\end{equation}

\textbf{Case 1.} 
If $\gamma = \Theta(1)$, then Eqs.~\eqref{eq:n-prime}, \eqref{eq:p-prime}, \eqref{eq:rewrite}, and $p=o(1)$ tell us that $p'$ is bounded inside the unit interval,
\[
1 - \e^{-\alpha \gamma} - o(1) 
\le p' 
\le 1 - \e^{-\beta \gamma} + o(1) \, .
\]
Since $|V_{t^*+1}|$ is binomial, the standard Chernoff bound implies that it is concentrated around its mean $n'p' = (1-o(1)) np'$, so \whp\
\begin{equation}
\label{eq:thing}
1-\e^{-\alpha \gamma} - o(1)
\le \frac{|V_{t^*+1}|}{n}
\le 1-\e^{-\beta \gamma} + o(1) \, ,
\end{equation}
so that $V_{t^*+1}$ contains a positive fraction of the vertices. The probability that a vertex in the rest of the graph does not have a neighbor in $V_{t^*+1}$ is then 
\[
(1-p)^{|V_{t^*+1}|} \le \e^{-p |V_{t^*+1}|} 
= \e^{-\Omega(d)} = o(1) \, ,
\]
so \whp\ $|V_{> t^*+2}| = o(n)$. Thus we have
\[
\frac{|V_{t^*+1}|}{n} + \frac{|V_{t^*+2}|}{n} 
= 1 - \frac{|V_{\le t^*}|}{n} - \frac{|V_{> t^*+2}|}{n}
= 1 - o(1) \, ,
\]
which with~\eqref{eq:thing} gives
\[
\e^{-\beta \gamma} - o(1) 
\le \frac{|V_{t^*+2}|}{n} 
\le \e^{-\alpha \gamma} + o(1) \, .
\]
This completes the proof for Case 1.

\textbf{Case 2.} If $\gamma = \omega(1)$, then~\eqref{eq:p-prime} and~\eqref{eq:rewrite} give
\[
p' \ge 1 - \e^{-\alpha \gamma} 
= 1-o(1) \, , 
\]
so $V_{t^*+1}$ contains almost all the remaining $n'$ vertices outside $V_{\le t^*}$. The number of vertices in the remainder $V_{> t^*+1}$, which by Lemma~\ref{lem:diameter} is \whp\ just $V_{t^*+2}$, is distributed as $\Bin(n',1-p')$ and has expectation at most 
\[
n'(1-p') \le n (1-p') \le n \e^{-\alpha \gamma} \, . 
\]

However, to apply the union bound over all $w$ we need a bound on $|V_{t^*+2}|$ that holds with probability $1-o(1/n)$. Since $\gamma = \omega(1)$, we have $\gamma/d = \omega(1/d) = \omega(1/\log^5 n)$. If $\e^{-\alpha \gamma} = \Omega(1/\log^5 n)$, then standard Chernoff bounds imply that $|V_{t^*+2}|/n \le (1+o(1)) \,\e^{-\alpha \gamma}$ with probability $1-o(1/n)$. Alternatively, if $\e^{-\alpha \gamma} = o(1/\log^5 n)$, Chernoff bounds let us absorb  $|V_{t^*+2}|/n$ into the error term $o(1)\,(\beta\gamma/d)$ in~\eqref{eq:case2-tplus1} with the same probability. Either way we have
\[
\frac{|V_{\le t^*}|+|V_{t^*+2}|}{n} \le (1+o(1)) \left( \frac{\beta\gamma}{d} + \e^{-\alpha \gamma} \right) \, ,
\]
which completes the proof of~\eqref{eq:case2-tplus1}.
\end{proof}

Finally, we combine Lemma~\ref{lem:last-few} with the entropic bound of Lemma~\ref{lem:entropy} to prove our lower bounds:

\begin{theorem}
\label{thm:main-lower}
Let $G = G(n,p)$ where $p=d/n$ and $c \log n \le d \le \log^5 n$ for some $c > 1$, and let $\alpha$, $\beta$, and $\gamma$ be as above. In Case 1 where $\gamma=\Theta(1)$, \whp\ 
\begin{align}
\text{if $d=O(\log n)$,}\quad \MD(G) &\ge
 (1-o(1)) \,\frac{\log n}{\max_{x \in [\alpha,\beta]} H(\e^{-x\gamma})} 
\label{eq:lower-case1-olog} \\
\text{if $d=\omega(\log n)$,}\quad \MD(G) &\ge (1-o(1)) \frac{\log n}{H(\e^{-\gamma})} \, ,
\label{eq:lower-case1-omegalog}
\end{align}
where $H(y) = -y \log y -(1-y) \log (1-y)$ is the binary entropy function. 
In Case 2 where $\gamma=\omega(1)$, \whp
\begin{align}
\text{if $d=O(\log n)$,}\quad \MD(G) &\ge 
(1-o(1)) \,\frac{\log n}{-(\beta \gamma/d) \log (\beta \gamma/d) 
+ \gamma \e^{-\alpha \gamma}} 
\label{eq:lower-case2-olog} \\
\text{if $d=\omega(\log n)$,}\quad \MD(G) &\ge (1-o(1)) 
\,\frac{\log n}{-(\gamma/d) \log (\gamma/d) 
+ \gamma \e^{-\gamma}} 
\, . 
\label{eq:lower-case2-omegalog}
\end{align}
\end{theorem}

\begin{proof}
Given Lemma~\ref{lem:entropy}, we simply have to show that \whp\  the entropy $H_w$ of distances for each $w$ is less than some $\Hmax$ for all $w$. The denominators in \eqref{eq:lower-case1-olog}--\eqref{eq:lower-case2-omegalog} are then essentially $\Hmax$ in each case. 

Recall that 
\[
H_w = - \sum_{t=0}^\Diam \frac{|V_t(w)|}{n} \log \frac{|V_t(w)|}{n} \, . 
\]
We first bound the contribution to $H_w$ from the shells $V_t$ for $t \le t^*$, each of which occupies $o(n)$ vertices. In both cases this contribution is $o(1)$, but in Case 2 it contributes importantly to the denominator. Since $-x \log x$ is monotonically increasing for $x < 1/\e$, Lemma~\ref{lem:expand} gives
\begin{align}
- \sum_{t=0}^{t^*} 
\frac{|V_t(w)|}{n} \log \frac{|V_t(w)|}{n} 
& \le
- (1+o(1)) \sum_{t=0}^{t^*} 
\frac{\beta d^t}{n} \log \frac{\beta d^t}{n} 
\nonumber \\
& \le 
- (1+o(1)) \,
\frac{\beta d^{t^*}}{n} \log \frac{\beta d^{t^*}}{n} 
\nonumber \\
& = -(1+o(1)) \,
\frac{\beta \gamma}{d} \log
\frac{\beta \gamma}{d} \, ,
\label{eq:hw-tstar}
\end{align}
where we used the fact that, analogous to~\eqref{eq:geometric}, the sum over $t$ is dominated by its largest term, namely the contribution from $V_{t^*}$. This contribution is $o(1)$  since $\gamma/d=o(1)$. 

In Case 1, Eqs.~\eqref{eq:case1-tplus1}, \eqref{eq:case1-tplus2}, and
\eqref{eq:case1-rest} in Lemma~\ref{lem:last-few} imply that the contribution from the two large shells is
\begin{align*}
&- \frac{|V_{t^*+1}(w)|}{n} \log \frac{|V_{t^*+1}(w)|}{n} 
- \frac{|V_{t^*+2}(w)|}{n} \log \frac{|V_{t^*+2}(w)|}{n} \\
&\le H\!\left( \frac{|V_{t^*+1}|}{n} \right) 
= H(\e^{-x\gamma}) 
\quad \text{for some $x \in [\alpha,\beta]$.} 
\end{align*}
By Lemma~\ref{lem:diameter} we have at most one more shell $V_{t^*+3}$, and since its size is $o(n)$ it contributes $o(1)$ to $H_w$. Thus we have shown that 
\whp\ in Case 1, $H_w \le \Hmax$ for all $w$ where
\[
\Hmax = \max_{x \in [\alpha,\beta]} H(\e^{-x \gamma}) + o(1) \, .
\]
This proves~\eqref{eq:lower-case1-olog}, and recalling from Lemma~\ref{lem:expand} that $\alpha=\beta=1$ when $d=\omega(\log n)$ proves~\eqref{eq:lower-case1-omegalog}.

In Case 2, the contribution~\eqref{eq:hw-tstar} of $V_{t^*}$ to $H_w$ matters, as does the contribution from $V_{t^*+2}$. As in the proof of Lemma~\ref{lem:last-few}, either the expectation of $|V_{t^*+2}|/n$ is $\Omega(1/d)=\Omega(n/\log^5 n)$, in which case Chernoff bounds ensure that $|V_{t^*+2}|/n \le (1+o(1)) \,\e^{-\alpha \gamma}$ with probability $1-o(1/n)$, or $|V_{t^*+2}|/n = o(1/d)$ with the same probability. In either case, \whp\ for all $w$ the contribution of $V_{t^*}$ and $V_{t^*+2}$ to $H_w$ can be bounded as
\begin{equation}
-\frac{V_{t^*}}{n} \log \frac{V_{t^*}}{n}
-\frac{V_{t^*+2}}{n} \log \frac{V_{t^*+2}}{n}
\le 
(1+o(1)) \left( 
-\frac{\beta\gamma}{d} \log \frac{\beta\gamma}{d}
+ \alpha \gamma \e^{-\alpha \gamma} \right) \, ,
\label{eq:hw-case2}
\end{equation}
where we pessimistically maximized the two terms separately.

Finally, since $-(1-p) \log (1-p) \le p$, the contribution of $V_{t^*+1}$ to $H_w$ is at most $(|V_{\le t^*}|+|V_{t^*+2}|)/n$, which due to the logarithmic terms is small compared to~\eqref{eq:hw-case2}. This completes the proof of~\eqref{eq:lower-case2-olog}, and again recalling that $\alpha=\beta=1$ when $d=\omega(\log n)$ gives~\eqref{eq:lower-case2-omegalog}.
\end{proof}

\section{The upper bound}
\label{sec:upper}

Our upper bound on $\MD(G(n,p))$ is essentially identical to that in~\cite{BollMitPra}. The only difference is that fluctuations in vertex degrees have to be taken into account when $d=O(\log n)$.

We follow a standard strategy. We will show that, with high probability, all pairs of vertices $u,v$ have a large number of vertices $w$ that separate them. Therefore, if we construct $W$ by choosing enough landmarks $w$ uniformly at random, there is a good chance that $W$ separates every pair. The following lemma is standard, but we include a proof for the reader.

\begin{lemma}
\label{lem:random-sep}
Let $G=(V,E)$ be a graph with $|V|=n$. For each distinct pair of vertices $u, v \in V$, let 
\[
S(u,v) = \{ w \in V : d(u,w) \ne d(v,w) \}
\]
denote the set of vertices $w$ that separate $u,v$. Suppose that for all  $u,v$ we have $S(u,v) \ge \sigma n$ for some $\sigma=\sigma(n)$. 
Then
\[
\MD(G) \le \left\lceil \frac{2 \log n}
{-\log (1-\sigma)} \right\rceil \, .
\]
\end{lemma}

\begin{proof}
Let $Z = \lceil (2 \log n)/(-\log (1-\sigma)) \rceil$ and let $W \subseteq V$ be a set of $Z$ of vertices chosen uniformly with replacement. For a given pair $u, v$, each $w \in W$ independently distinguishes them with probability at least $\sigma$. Therefore, the probability that $u$ and $v$ are not separated by any $w \in W$ is at most 
$(1-\sigma)^Z 
\le \e^{-2 \log n}
= 1/n^2$. 
Taking the union bound over the $\binom{n}{2}$ distinct pairs, $W$ is a separator with probability at least $1-\binom{n}{2}/n^2 \ge 1/2$. Thus some separator $W$ of size $Z$ exists, and $\MD(G) \le Z$.
\end{proof}

We will apply Lemma~\ref{lem:random-sep} as follows. For a given $t$ if we write
\[
\Delta_t(u,v) = V_t(u) \,\symdef\, V_t(v)
\]
where $\symdef$ denotes the symmetric difference, then
\begin{equation}
\label{eq:S_union}
S(u,v) = \bigcup_{t=0}^{\Diam} 
\Delta_t (u,v) \, . 
\end{equation}
To lower bound $|S(u,v)|$, we first show that after $t^*$ steps of expansion, the shells for any two distinct vertices are nearly disjoint, so that their symmetric difference is almost as large as their union: 

\begin{lemma}
\label{lem:from-a-pair}
Let $G = G(n,p)$ where $p=d/n$ and $c \log n \le d \le \log^5 n$ for some $c > 1$, and let $t^*$, be defined as in Definition~\ref{def:cases}. Then \whp\ the following holds for all distinct pairs of vertices $u, v \in V$:
\begin{align}
|\Delta_{t^*}(u,v)|
&= (1 \pm o(1)) \,d^{t^*-1} (\deg u + \deg v) 
\nonumber \\
&= (1 \pm o(1)) 
\,\frac{2 \alpha \gamma n}{d} \, . 
\label{eq:from-a-pair}
\end{align}
\end{lemma}

\noindent We prove Lemma~\ref{lem:from-a-pair} in Section~\ref{sec:proof-expand}  by combining the proof of Lemma~\ref{lem:expand} with the fact that \whp\ no two vertices have more than two common neighbors. 

Lemma~\ref{lem:from-a-pair} establishes that $u$ and $v$ have a fairly large number of separators in their $t^*$th shells. They might have even more in their $(t^*+1)$st shell: in fact, in Case 1 we will show that  $|\Delta_{t^*+1}| = \Theta(n)$. Thus, analogous to Lemma~\ref{lem:last-few}, the next lemma lower bounds $|S(u,v)|$ either by including both these shells or just the $(t^*+1)$st, depending on the case. 

\begin{lemma}
\label{lem:separators}
    Let $G=G(n,p)$ where $p=d/n$, and let $d, c, t^*,\gamma,\alpha,\beta$ be as above, with $c \log n \le d \le \log^5 n$ for some $c > 1$. In Case 1 where $\gamma=\Theta(1)$, \whp\ for all distinct $u,v \in V$:
    \begin{align}
        \frac{|S(u,v)|}{n} 
        & \ge \frac{|\Delta_{t^*+1}(u,v)|}{n}
        \nonumber \\
        & \ge (2-o(1))
        \,(1-\e^{-\alpha\gamma}) 
        \,\e^{-\beta\gamma} 
        \, . 
        \label{eq:sym_dif_bound-case1}
    \end{align}
In Case 2 where $\gamma=\omega(1)$, \whp\ for all distinct $u,v \in V$:
\begin{align}
    \frac{|S(u,v)|}{n} 
        & \ge \frac{|\Delta_{t^*}(u,v)| + |\Delta_{t^*+1}(u,v)|}{n}
        \nonumber \\
        & \ge (2-o(1)) \left( 
        \frac{\alpha \gamma}{d} + 
        \e^{-\beta\gamma} 
        \right) 
        \, . 
        \label{eq:sym_dif_bound-case2}
    \end{align}
\end{lemma}

\begin{proof}
    A vertex is in $\Delta_{t^*+1}(u,v)$ 
    if it is outside $V_{\le t^*}(u) \cup V_{\le t^*}(v)$ and has an edge to either $V_{t^*+1}(u)$ or $V_{t^*+1}(v)$ but not both. Similar to~\eqref{eq:n-prime} in the proof of Lemma~\ref{lem:last-few}, the number of eligible vertices is
    \[
    n'=n-|V_{\le t^*}(u)\cup V_{\le t^*}(v)|
    = (1-o(1)) \,n \, .
    \]
    Then conditioned on $n'$, $|V_{t^*}(u)|$, $|V_{t^*}(v)|$, and $|V_{t^*}(u) \cap |V_{t^*}(v)|$, we have
\begin{equation}
     |\Delta_{t^*+1}(u,v)| \sim \Bin(n',p')
\end{equation}
     where
\begin{align*}
    p' = &
    \left(1-(1-p)^{|V_{t^*}(u) \setminus V_{t^*}(v)|} \right)
    (1-p)^{|V_{t^*}(v)|} \\
    + & 
    \left(1-(1-p)^{|V_{t^*}(v) \setminus V_{t^*}(u)|}\right)(
    1-p)^{|V_{t^*}(u)|} \, .
\end{align*}
Lemma~\ref{lem:from-a-pair} implies that 
\[
|V_{t^*}(u) \cap |V_{t^*}(v)| = o(|V^{t^*}(u)|) \, ,
\]
and so 
\[
(1-p)^{|V_{t^*}(u) \setminus V_{t^*}(v)|} 
= (1-p)^{|V_{t^*}(u)|} + o(1) \, . 
\]
Since by Lemma~\ref{lem:expand} we have
\[
(1-p)^{|V_{t^*}(u)|} 
\le \e^{-p |V_{t^*}(u)|} 
\le (1+o(1)) \,\e^{-\alpha \gamma} \, ,
\]
it follows that
\[
1-(1-p)^{|V_{t^*}(u) \setminus V_{t^*}(v)|}
= (1-o(1)) 
\left( 1-(1-p)^{|V_{t^*}(u)|} \right)
\, ,
\]
and similarly for the term with $u$ and $v$ swapped. Then using~\eqref{eq:rewrite} and the Taylor series as in the proof of Lemma~\ref{lem:last-few} gives
\begin{equation}
\begin{aligned}
p'
&= (1-o(1)) \left[ \left(1-(1-p)^{|V_{t^*}(u)|}\right)(1-p)^{|V_{t^*}(v)|} \right. \\
&\qquad + \left. \left(1-(1-p)^{|V_{t^*}(v)|}\right)(1-p)^{|V_{t^*}(u)|} \right] \\
    &= (1-o(1)) \left[ \left(1-\e^{-p|V_{t^*}(u)|}\right) \e^{-p|V_{t^*}(v)|}
    + \left(1-\e^{-p|V_{t^*}(v)|} \right) \e^{-p|V_{t^*}(u)|} \right] \\
    &\ge (2-o(1))
    \,(1-\e^{-\alpha\gamma}) \,\e^{-\beta\gamma} 
    \, ,
    \end{aligned}
    \end{equation}
where we pessimistically minimized the two terms in the penultimate line separately. 
Thus the conditional expectation of $|\Delta_{t^*+1}|$ is bounded below by
    \[
        \nu 
        := n' p'
        \ge (2-o(1)) 
        \,(1-\e^{-\alpha\gamma}) \,\e^{-\beta\gamma} n 
        \, .
    \]
    
\textbf{Case 1.} If $\gamma=\Theta(1)$, then  $\nu=\Theta(n)$ and Chernoff bounds imply~\eqref{eq:sym_dif_bound-case1}.

\textbf{Case 2.} Since $\gamma=\omega(1)$ we have $\e^{-\alpha\gamma}=o(1)$, so we simply write $\nu = (2-o(1)) \,\e^{-\beta\gamma} n$. 

Now the first term of~\eqref{eq:sym_dif_bound-case2} corresponds to $\Delta_{t^*}$ and comes from Lemma~\ref{lem:from-a-pair}. For the second term involving $\Delta_{t^*+1}$, we have a situation similar to Case 2 of the lower bound. If $\nu \ge n/\log^{10} n$, then Chernoff bounds imply that $|\Delta_{t^*+1}| \ge (1-o(1)) \,\nu$ with probability $1-o(1/n^2)$. Alternatively, if  $\nu < n/\log^{10} n$, Chernoff bounds give $|\Delta_{t^*+1}| \le \nu \log^5 n = o(\gamma/d)$ with probability $1-o(1/n^2)$. In that case, we absorb $|\Delta_{t^*+1}|/n$ into the error term of $|\Delta_{t^*}|/n = (2-o(1)) \,\alpha \gamma / d$. 

In either situation we can take the union bound over all $u,v \in V$, completing the proof of~\eqref{eq:sym_dif_bound-case2}. 
\end{proof}

Finally, combining Lemmas~\ref{lem:separators} and~\ref{lem:random-sep} gives the following upper bounds on $\MD(G)$. In the following the function $q$ mirrors that in Theorem~\ref{thm:bollmitpra} from~\cite{BollMitPra}.

\begin{theorem}
\label{thm:main-upper}
Let $G=G(n,p)$ where $p=d/n$ and $c\log n \le d \le \log^5 n$ for some $c>1$, and let $\alpha,\beta$, and $\gamma$ be as above.

In Case 1 where $\gamma=\Theta(1)$, \whp
\begin{align}
\text{if $d=O(\log n)$,}\quad \MD(G) &\le 
(1+o(1))\,\frac{2 \log n}{-\log q(\alpha,\beta)} \, ,
\label{eq:upper-case1-olog} \\
\text{if $d=\omega(\log n)$,}\quad \MD(G) &\le (1+o(1)) 
\,\frac{2 \log n}{-\log q}
\, . 
\label{eq:upper-case1-omegalog}
\end{align}
where 
\begin{align*}
q(\alpha,\beta) 
&= 
1 - 2
\,(1-\e^{-\alpha\gamma}) 
\,\e^{-\beta\gamma} 
\\
q = q(1,1) &= 1-2(1-\e^{-\gamma}) \,\e^{-\gamma} 
= (\e^{-\gamma})^2 + (1-\e^{-\gamma})^2 \, .
\end{align*} 
In Case 2, \whp
\begin{align}
\text{if $d=O(\log n)$,}\quad \MD(G) &\le (1+o(1)) 
\,\frac{\log n}{\alpha\gamma/d + \e^{-\beta\gamma}} \, ,
\label{eq:upper-case2-olog} \\
\text{if $d=\omega(\log n)$,}\quad \MD(G) &\le  (1+o(1)) 
\,\frac{\log n}{\gamma/d + \e^{-\gamma}} \, .
\label{eq:upper-case2-omegalog}
\end{align}

\end{theorem}

\begin{proof}
    In Case 1, we apply Lemma~\ref{lem:random-sep} with the lower bound given by Lemma~\ref{lem:separators}, 
    \[
    \sigma = (2-o(1)) 
    \,(1-\e^{-\alpha\gamma})
    \,\e^{-\beta\gamma} 
    \]
    and write $q(\alpha,\beta)=1-\sigma$ (absorbing the error term).

    In Case 2, we likewise apply Lemma~\ref{lem:random-sep} but with the lower bound
    \[
    \sigma 
    = (2-o(1)) \left(\frac{\alpha\gamma}{d}+(1-\e^{-\gamma\alpha}) \,\e^{-\gamma\beta}\right)
    \, . 
    \]
    Moreover, in this case we have $\sigma=o(1)$ since $\gamma/d = o(1)$ and $\gamma = \omega(1)$. Then we use $-\log(1-\sigma) = (1+o(1)) \,\sigma$ for the denominator in Lemma~\ref{lem:random-sep}.
    
    In both cases, as per Corollary~\ref{cor:err_prob_deg}, when $\gamma=\omega(1)$ and $d=\omega(\log n)$ we can take $\alpha=\beta=1$, yielding~\eqref{eq:upper-case1-omegalog} and~\eqref{eq:upper-case2-omegalog}. In particular, $q=q(1,1)$ coincides with the quantity $q$ used in Theorem~\ref{thm:bollmitpra} from~\cite{BollMitPra}.
\end{proof}

Along with Theorem~\ref{thm:main-lower}, this completes the proof of our main Theorem~\ref{thm:main-combined}.

\section{Expansion: Proof of Lemmas~\ref{lem:expand} and~\ref{lem:from-a-pair}} 
\label{sec:proof-expand}

First we state two versions of the Chernoff bound. The first one is slightly non-standard, but it follows from the usual argument involving the moment generating function. 

\begin{lemma}
\label{lem:chernoff-1}
Let $X$ be a binomial random variable with mean $\mu$. Let $\delta, \kappa > 0$ be constants with $\delta < 1$.  Then there exist constants $0 < \alpha < 1 < \beta$ such that 
\begin{align}
\Pr[X < \alpha \mu] 
&\le \e^{-\delta \mu} \label{eq:chernoff-1-lower} \\
\Pr[X > \beta \mu] 
&\le \e^{-\kappa \mu} \, . 
\label{eq:chernoff-1-upper}
\end{align}
\end{lemma}

\begin{proof}
The moment generating function of $X \sim \Bin(n,p=\mu/n)$ is upper bounded by that of a Poisson variable with mean $\mu$, 
\[
\Exp[\e^{\lambda X}] = \left( 1 + p (\e^\lambda-1) \right)^n \le \e^{\mu (\e^\lambda-1) } \, . 
\]
Let $0 < \alpha < 1$. For any $\lambda < 0$, Markov's inequality gives
\[
\Pr[X < \alpha \mu] 
= \Pr[\e^{\lambda X} > \e^{\lambda \alpha \mu}] 
\le \Exp[\e^{\lambda X}] / \e^{\lambda \alpha \mu} 
= \e^{-\mu (1 - \e^\lambda + \lambda \alpha)} \, , 
\]
and respectively for $\Pr[X > \beta \mu]$ with $\lambda$ replaced by $-\lambda$. The right-hand side is minimized when $\lambda = \log \alpha$ (resp.\ $-\log \alpha$), giving
\begin{equation}
\label{eq:lower-tail}
\Pr[X < \alpha \mu] \le \e^{-\mu f(\alpha)} 
\quad \text{and} \quad
\Pr[X > \beta \mu] \le \e^{-\mu f(\beta)} 
\end{equation}
where
\[
f(\alpha) = 1 - \alpha + \alpha \log \alpha \, .
\]

The function $f$ is continuous with $f(0)=1$ and $f(1)=0$, so for any $0 < \delta < 1$ there is some $0 < \alpha < 1$ such that $f(\alpha) = \delta$. Similarly, $f(\beta)$ grows without bound for $\beta > 1$, so for any $\kappa > 0$ there is a $\beta$ such that $f(\beta) = \kappa$. 
\end{proof}

Applying this to the degree of a vertex in $G(n,p)$ yields the following, which proves~\eqref{eq:degree-bound} in Lemma~\ref{lem:expand}.
\begin{corollary}
\label{cor:err_prob_deg}
Let $G=(n,p)$ where $p=d/n$ and $d \ge c \log n$ for some $c = c(n) > 1$. Then there are constants $0 < \alpha < 1 < \beta$ such that \whp\ for all vertices $w$,
\begin{equation}
\label{eq:all-deg}
(1-o(1)) \,\alpha d 
\le \deg w 
\le (1+o(1)) \,\beta d \, .
\end{equation}
Furthermore, if $c=\omega(1)$ so that $d=\omega(\log n)$, \eqref{eq:all-deg} holds with $\alpha=\beta=1$.
\end{corollary}

\begin{proof}
If $X = \deg w$ for a given $w$, then $X \sim \Bin(n-1,p)$ and has mean $\mu = (1-1/n) d = (c-o(1)) \log n$. In Lemma~\ref{lem:chernoff-1} we set $\delta = \kappa = 1/\sqrt{c} < 1$, say. Then since $c >1 $, the error probabilities in~\eqref{eq:chernoff-1-lower} and~\eqref{eq:chernoff-1-upper} are 
\[
\e^{-\sqrt{c} \mu} 
= n^{-(1-o(1)) \sqrt{c}} 
= o(1/n) \, ,
\]
letting us take the union bound over all $w$. 

If $c=\omega(1)$, the Taylor series 
\[
f(1 + \eps) = \eps^2/2 - O(\eps^3) \, . 
\]
implies that we can set $\alpha=\beta=1$ and absorb $\eps$ in the $o(1)$ terms in~\eqref{eq:all-deg}. 
\end{proof}

\begin{remark}
In the regime $d = c \log n$ for constant $c$, the standard Chernoff bound does not imply that the minimum degree is $\Omega(\log n)$ until $c$ is some constant strictly larger than $1$, i.e., well above the connectivity threshold. Hence our use of the more precise Lemma~\ref{lem:chernoff-1}.    
\end{remark}

Since $|V_1(w)| = \deg w$, Corollary~\ref{cor:err_prob_deg} forms the base case for our inductive proof of Lemma~\ref{lem:expand}. For the inductive step we will use the usual Chernoff bound in the following form:

\begin{lemma}
\label{lem:chernoff-2}
Let $X$ be a binomial random variable with mean $\mu$. Then for any $0 < \eps < 1$, 
\begin{equation}
\label{eq:chernoff-standard}
\Pr\!\left[ 
\left| \frac{X}{\mu} - 1 \right| > \eps \right] 
\le 2 \e^{-\eps^2 \mu / 3} \, .
\end{equation}
In particular, if $\mu = \Omega(\log^t n)$ for $t \ge 2$, then there is a constant $A$ such that
\begin{equation}
\label{eq:chernoff-2}
(1-\eps) \mu \le X \le (1+\eps) \mu
\quad \text{where} \quad
\eps = \frac{A}{\log^{(t-1)/2} n}
\end{equation}
with probability $1-o(1/n^3)$.
\end{lemma}

\begin{proof}
We define $\eps$ in~\eqref{eq:chernoff-2} so that $\eps^2 \mu = \Omega(\log n)$. If $\mu \ge C \log^t n$, then setting $A=10/C$ suffices to make the error probability in~\eqref{eq:chernoff-standard} $o(1/n^3)$.
\end{proof}

We now prove Lemma~\ref{lem:expand} by recursively applying Lemma~\ref{lem:chernoff-2}, using Corollary~\ref{cor:err_prob_deg} as the base case.

\begin{proof}[Proof of Lemma~\ref{lem:expand}.] 
Since Corollary~\ref{cor:err_prob_deg} established that~\eqref{eq:degree-bound} holds with high probability, we condition on this event. In  particular we assume that $\deg w = \Omega(\log d)$ for all $w$. Our goal is then to prove that \eqref{eq:Nupper-lower-expand0} holds for all $1 \le t \le t^*$ with probability $1-o(1/n)$ for any fixed $w$, and therefore holds \whp\ for all vertices $w$. 

Recall that, at each step of expansion, if we condition on $|V_{t'}|$ for $t' < t$ then $|V_t|$ is a binomial random variable. We denote its expection $\mu_t$. There are two sources of multiplicative error that we need to track over the $t^*$ steps of expansion. First, $\mu_t$ differs slightly from $d |V_{t-1}|$, and second, $|V_t|$ differs slightly from $\mu_t$. The first source of error increases as $t$ approaches $t^*$, and the second decreases as $t$ and $\mu_t$ increase and the Chernoff bound becomes tighter. Happily, the product of all these error terms is $1 \pm o(1)$.

Very similar reasoning appears in Lemma 2.1 of~\cite{BollMitPra}, which they prove for $d=\omega(\log n)$. Since we include the case $d=O(\log n)$, the vertex degrees, and hence the sizes of $V_t$, have significant fluctuations as in Corollary~\ref{cor:err_prob_deg}, but this only affects the base case of the induction at $t=1$. 
(This base case will also change in Lemma~\ref{lem:from-a-pair}, where we consider $V_t(u) \cup V_t(v)$ for a pair of vertices $u,v$.)

For each $1 \le t \le t^*$, let $E_t$ denote the event
\[
\label{eq:Et_statement}
\Pi^-_t d^{t-1} \deg w 
\le |V_t(w)|
\le \Pi^+_t d^{t-1} \deg w
\]
where, for some $\{\delta_t\}$ and $\{\eps_t\}$ to be defined below,
\[
\Pi^{-}_t = \prod_{t'=2}^{t}
(1 - \delta_{t'})(1 - \eps_{t'}) 
\quad \text{and} \quad
\Pi^{+}_t = \prod_{t'=2}^{t}
(1 + \eps_{t'}) \, . 
\]
(Of course, $\Pi^-_1=\Pi^+_1=1$.) 
We also write $E_{\le t} = \bigcap_{t'=1}^t E_{t'}$ and $E_{<t}$ similarly. 

Since $|V_1(w)| = \deg w$, the base case $E_1$ holds trivially. Thus
\begin{equation}
\label{eq:prod_form_Eu}
\Pr[E_{\le t^*}] 
= \Pr\!\left[ \,\bigcap_{t=2}^{t^*} E_t \right]
= \prod_{t=2}^{t^*} \Pr[E_t \mid E_{<t}] 
\, .
\end{equation}
We will show that 
\[ 
\Pr[E_t \mid E_{<t}] = 1-o(1/n^3) \, ,
\]
so we can take the union bound over all $t^* < \log n$ steps of the induction, as well as over all $n$ vertices $w$ (and, in Lemma~\ref{lem:from-a-pair}, over all $n^2$ pairs of vertices $u,v$).  Specifically, for each $2 \le t \le t^*$ we will show that, conditioned on $E_{<t}$, the expectation $\mu_t$ of $|V_t|$ is deterministically almost $d |V_{t-1}|$.
\begin{equation}
\label{eq:mult1}
(1-\delta_t) \,d |V_{t-1}| 
\;\le\; 
\mu_t 
\;\le\;
d |V_{t-1}|
\end{equation}
We will also that that, with probability $1-o(1/n^3)$,
\begin{equation}
\label{eq:mult2}
(1-\eps_t) \,\mu_t 
\;\le\; |V_t| 
\;\le\; (1+\eps_t) \,\mu_t \, . 
\end{equation}
Together these imply the induction step
\begin{equation}
\label{eq:induction}
(1-\delta_t)(1-\eps_t) \,d |V_{t-1}|
\le |V_t| 
\le (1+\eps_t) \,d |V_{t-1}| \, ,
\end{equation}
and therefore establish $E_t$. Finally, the values of $\delta_t$ and $\eps_t$ that we will give below in Eqs.~\eqref{eq:delta} and \eqref{eq:eps} decrease geometrically as we count backwards from $t^*$ or forwards from $t=2$ respectively, so that the products $\Pi^\pm_t$ are $1 \pm o(1)$. 

To prove~\eqref{eq:mult1} and~\eqref{eq:mult2}, let $t \le t^*$ and suppose inductively that $E_{<t}$ holds. As in the proof of Lemma~\ref{lem:last-few}, 
\[
|V_t| \sim \Bin(n',p')
\quad \text{with mean} \quad
\mu_t=n'p' \, ,
\]
where 
\[
n'=n-|V_{<t}|
\quad \text{and} \quad
p'=1-(1-p)^{|V_{t-1}|} \, . 
\]
Clearly $\mu_t \le d |V_{t-1}|$ since each vertex in $V_{t-1}$ has $d$ neighbors in expectation. To see that $\mu_t$ is not much smaller than this, note that since $E_{<t}$ holds $|V_{<t}|$ is a geometric sum dominated by its largest term, i.e.,  
\[
|V_{<t}| 
= O(|V_{t-1}|)
= O(d^{t-1}) 
= O(d^{t^*} d^{-(t^*-t+1)})
= o(d^{-(t^*-t+1)} n) \, ,
\]
where we used $d^{t^*}=o(n)$. Thus
\begin{equation}
\label{eq:nprime-ind}
n' = \big( 1-o(d^{-(t^*-t+1)}) \big) \,n \, .
\end{equation}
For $p'$, the error term comes from having more than one neighbor in $V_{t-1}$. Taking the Taylor series gives
\begin{align}
p' 
&= \big( 1-O(p |V_{t-1}|) \big) \,p |V_{t-1}| 
\nonumber \\
&= \big( 1-O(d^t/n) \big) \,p |V_{t-1}| 
\nonumber \\
&= \left( 
1-O\big( (\gamma/d) \,d^{-(t^*-t)} \big) 
\right) 
p |V_{t-1}| \, ,
\label{eq:pprime-ind}
\end{align}
where we used $d^{t^*}/n = \gamma/d$. 
Combining~\eqref{eq:nprime-ind} and~\eqref{eq:pprime-ind} gives
\[
\mu_t = n'p' = \left( 
1-O\big( (\gamma/d) \,d^{-(t^*-t)} \big) 
\right) 
\,d |V_{t-1}| \, .
\]
This proves~\eqref{eq:mult1} with
\begin{equation}
\label{eq:delta}
\delta_t = B (\gamma/d) \,d^{-(t^*-t)} 
\end{equation}
for some constant $B > 0$. 

To prove that~\eqref{eq:mult2} holds with probability $1-o(1/n^3)$, we invoke the Chernoff bound in Lemma~\ref{lem:chernoff-2}. Conditioning on $E_{<t}$ implies $\mu_t = \Theta(d^t) = \Omega(\log^t n)$, so~\eqref{eq:chernoff-2} implies that~\eqref{eq:mult2} holds with probability $1-o(1/n^3)$, where 
\begin{equation}
\label{eq:eps}
\eps_t = A (\log n)^{-(t-1)/2} 
\end{equation}
for some constant $A > 0$.

Finally, we have
\begin{align*}
-\log \Pi^{-}_t 
&= -\sum_{t'=2}^t \log ( 1 - \delta_{t'} ) 
- \sum_{t'=2}^t \log ( 1 - \eps_{t'} ) \\
&= -O\!\left(
\sum_{t'=2}^t \delta_{t'} 
+ \sum_{t'=2}^t \eps_{t'} 
\right) \, .
\end{align*}
Both these sums are geometric and are dominated by their largest term. Thus for any $2 \le t \le t^*$, recalling $\gamma/d = o(1)$ gives
\[
-\log \Pi^{-}_t 
= O(\delta_t + \eps_2) \\
= O\big( (\gamma/d) \,d^{-(t^*-t)} + (\log n)^{-1/2} \big) 
= o(1) \, .
\]
This implies $\Pi^-_t = 1-o(1)$, and similarly for $\Pi^+_t$. This completes the proof of~\eqref{eq:Nupper-lower-expand0}, which along with Corollary~\ref{cor:err_prob_deg} yields~\eqref{eq:Nupper-lower-expand}.
\end{proof}

\begin{proof}[Proof of Lemma~\ref{lem:from-a-pair}] Following~\cite{BollMitPra}, we first show that with high probability, for all distinct $u, v\in V$,
\[
|V_1(u) \cap V_1(v)| \le 2 \, ,
\]
i.e., no two vertices have more than two common neighbors. This follows from the classic fact that sparse random graphs \whp\ have no small subgraphs with more than one loop, i.e., with more edges than vertices.
A pair $u,v$ with three common neighbors forms a subgraph with $5$ vertices and $6$ edges, namely a copy of a complete bipartite graph $K_{2,3}$. The expected number of such pairs is 
\[
\binom{n}{2} 
\binom{n-2}{3} \,p^6 
= O(n^5 p^6)
= O(d^6/n)
= o(1) \, .
\]
Thus, by the first moment method, \whp\ no such pair exists. 

This implies \whp\ for all distinct $u,v$
\[
\deg u + \deg v - 2
\;\le\; |V_1(u)\cup V_1(v)|
\;\le\; \deg u + \deg v \, , 
\]
and in particular 
\[
|V_1(u) \cup V_1(v)|
= (1-o(1)) (\deg u + \deg v) \, .
\]
For each $t$, we can view
$V_t(u) \cup V_t(v)$ as the $t$th shell around the set $\{u,v\}$. As in the proof of Lemma~\ref{lem:expand}, \emph{mutatis mutandis} the main fluctuations in the size of these shells are due to the first step, and \whp\ the remaining steps each expand by a factor very close to $d$. Thus we have \whp\ for all $u,v$,
\begin{align*}
|V_{t^*}(u) \cup V_{t^*}(v)|
&= (1 \pm o(1)) d^{t^*-1} 
|V_1(u) \cup V_1(v)| \\
&= (1 \pm o(1)) \,d^{t^*-1}
(\deg u + \deg v) \, .
\end{align*}
Along with Lemma~\ref{lem:expand} this implies that $V_{t^*}(u)$ and $V_{t^*}(v)$ have a small intersection, i.e., 
\begin{align*}
|V_{t^*}(u) \cap V_{t^*}(v)| 
&= |V_{t^*}(u)| 
+ |V_{t^*}(v)| 
- |V_{t^*}(u) \cup V_{t^*}(v)|
\\
&= o(1) \,d^{t^*-1}
(\deg u + \deg v)
\, ,
\end{align*}
and that the symmetric difference is almost as large as the union, 
\begin{align*}
|\Delta_{t^*}(u,v)| 
&= |V_{t^*}(u) \cup V_{t^*}(v)| 
- |V_{t^*}(u) \cap V_{t^*}(v)| 
\\
&= (1 \pm o(1)) \,d^{t^*-1}
(\deg u + \deg v) \, .
\end{align*}
Finally, Corollary~\ref{cor:err_prob_deg} gives $\deg u + \deg v \ge (2-o(1)) \alpha d$, whereupon $d^{t^*} = \gamma n/d$ gives~\eqref{eq:from-a-pair}. 
\end{proof}

\end{document}